\documentclass[12 pt, a4paper]{article}
\usepackage[T1]{fontenc}

\usepackage{xcolor}
\usepackage{pdfcolmk}
\usepackage{amsthm}
\usepackage{amsmath}
\usepackage{amstext}
\usepackage{graphicx}

\usepackage{amsfonts}
\usepackage{amssymb}
\usepackage{amssymb,latexsym,epsfig,color}
\usepackage{enumerate}
\usepackage[all,cmtip]{xy}
\usepackage{cite}

\usepackage{amssymb}
\usepackage{setspace} 
\newtheorem{theorem}{Theorem}
\newtheorem{question}{Question}
\newtheorem{lemma}[theorem]{Lemma}
\newtheorem{remark}[theorem]{Remark}

\newtheorem{corollary}[theorem]{Corollary}

\newtheorem{definition}[theorem]{Definition}

\newcommand{\N}{{\mathbb N}}
\newcommand{\R}{{\mathbb R}}

\newcommand{\boldm}{{\bf M}}
\newcommand{\boldc}{{\bf C}}

\newcommand{\btop}{{\bf Top}}
\newcommand{\boldt}{{\bf T}}

\newcommand{\bm}{{\textbf M}}

\newcommand{\A}{{\mathcal A}}

\newcommand{\D}{{\mathcal D}}

\title{Complete Regularity {\it \`{a} la quantale}.}
\author{Jorge Bruno}
\date{University of Winchester\\
 jorge.bruno@winchester.ac.uk}
\begin{document}
\maketitle

\begin{abstract} Nearly three decades from his celebrated result, we study a modern refinement and strengthening of Kopperman's full metrisabilty of all topological spaces. Within this new theory of \emph{$V$-spaces}, developed by Flagg and Weiss, we investigate several topological notions and their metric counterparts. Among our main results is the reconstruction, in terms of $V$-spaces, of Kopperman's equivalence between symmetric continuity spaces valued in a value
semigroup and completely regular topologies. We conclude our work by revisiting some classical topological results and their almost evident validity through this metric lens. 
\end{abstract}

\section{Introduction}

 Let ${\bf Met}$ denote the of category metric spaces with $\epsilon$-$\delta$ continuous functions. The coincidence of both types of morphisms implies that the obvious functor ${\bf Met} \to \btop$, that sends any metric space to the topological space it generates, is fully faithful on $\btop_\boldm$ - the category of metrisable topological spaces. The fullness of this functor is largely due to the triangle inequality imposed on any metric space. Were it not the case, it would only be true that $\epsilon$-$\delta$ continuity implies topological continuity. Indeed, take $X = \{a,b,c,e\}$ with $d: X^2 \to \R$ as
 \[
 d(x,y) = 
 \begin{cases}
 0 & \text{ if } \{x,y\} = \{a,b\}\text{ or } \{x,y\} = \{b,c\},\\
 2 & \text{ if } \{x,y\} = \{a,c\}, \text{ and}\\
1 & \text{ otherwise}.
 \end{cases}
 \]

\noindent
The reader can quickly verify that int$[B_2(a)] = \{e\}$. It thus follows that $\epsilon$-$\delta$ continuous functions are topologically continuous but the converse is certainly not true. Back to $ {\bf Met} \to \btop$, it is then most natural to ask for a fully faithful extension $\mathcal{O}: \boldc \to \btop$ of $ {\bf Met} \to \btop$ so that the diagram  
 
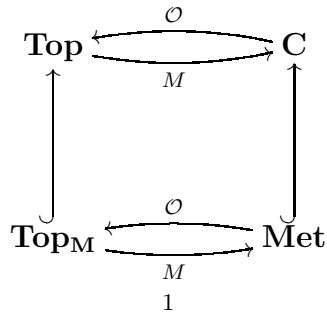
\[
 \xymatrix{ 
  {\bf Top} \ar@/_/[rr]_M & &  {\bf C} \ar@/_/[ll]_{\mathcal{O}}  \\
  &&\\
  {\bf Top_M}  \ar@{^{(}->}[uu]\ar@/_/[rr]_M & &  {\bf Met} \ar@/_/[ll]_{\mathcal{O}} \ar@{^{(}->}[uu]\\
  }
\]
 
\noindent 
commutes and where the pair of functors $(\mathcal{O},M)$ represent an equivalence of categories. That is, a comprehensive and cohesive \emph{metrisation} of all topological spaces where topological notions can be naturally interpreted by their metric counterparts. 

The first attempt at creating this extension was achieved by Kopperman in terms of his \emph{continuity spaces} - spaces valued in value semigroups. Kopperman's theory captures some of the properties from $([0,\infty],\leq,+)$ that make metric techniques more powerful than topological ones. A striking fact of this metrisation is the correspondence, illustrated by Kopperman in  \cite{MR935419}, between \emph{symmetric} continuity spaces and their completely regular counterparts. A result that further illustrates the naturality of this metrisation; one would be inclined to expect such an equivalence given the well-known correspondence between completely regular topologies and uniform spaces.

In Kopperman's theory the concept of positive elements, naturally occurring in $([0,\infty],\leq,+)$, is not an intrinsic one. A refinement of Kopperman's approach is the one initiated by Flagg in \cite{MR1452402} based on his quantale valued \emph{$V$-spaces}. Given a value quantale $V$, a $V$-space is a pair $(X,d)$ where $X$ is any set and $d:X\times X \to V$ is a distance assignment to pairs of points from $X$ valued in $V$ so that:
\begin{itemize}
\item $d(x,x)=0$ for all $x\in X$, and
\medskip 
\item $d(x,z)\le d(x,y)+d(y,z)$ for all $x,y,z\in X$.
\end{itemize} 
  In \cite{MR3334228} Flagg's $V$-spaces are then captured by Weiss as objects within a suitable category, which he denotes by ${\bf M}_{\bf T}$. Inspired by the equivalence between $\epsilon$-$\delta$ continuity and its topological form, Weiss encompassed topological continuity in terms of a binary relation naturally present in value distributive lattices: the \emph{well-above relation}. The reader interested in the interplay between both types of metrisations is referred to \cite{IttayJPAA}. Within the last few years this latter metrisation has received a fair amounts of attention: in \cite{MR3482717} a metric characterisation of topological connectedness is constructed; Ackerman extends Banach's fixed point theorem for a large class of symmetric $V$-spaces in \cite{MR3346937}; the authors of \cite{MR3310342} describe a generalisation of the classical completion of a metric space in terms of $V$-spaces and the concept of \emph{uniformly vanishing asymmetry}; in \cite{Bruno2017} sequences are studied in a setting where the triangle inequality is dropped;  while a categorical framework for topological invariants is developed in \cite{InvariantIttay}. Other work, past and present, pertaining to this topic can be found in \cite{LAI2017599}, \cite{Lawvere1973} and \cite{lmcs:3861}, and an excellent encyclopaedic reference can be found in \cite{hofmann2014monoidal}.
 
In light of the above work, the aim of this paper is to further augment the study of Flagg's metrisation. More precisely, we are concerned with reconstructing some classical topological notions in ${\bf M}_{\bf T}$. We begin by interpreting basic topological constructs (e.g., product spaces, box topologies, quotients, topological sums, etc.) in ${\bf M}_{\bf T}$ and illustrating their complexity within this metric context. We then turn our attention to our main result: reproducing Kopperman's correspondence in terms of Flagg's $V$-spaces. That is, constructing an equivalence between symmetric $V$-spaces and completely regular topologies and thus further evidence Flagg's metric approach to topology as the correct refinement of Kopperman's. With the above results at our disposal, we conclude our work by revisiting some classical topological results and their almost evident validity when interpreted through this metric lens. In particular, the latter assertion being a consequence of the intrinsic complexity of topological (co)limits in ${\bf M}_{\bf T}$.

 \section{Background}
 
The construction of objects from ${\bf M}_{\bf T}$ can be found in \cite{MR1452402} (pg. 268) while the equivalence is illustrated in \cite{MR3334228}. For completion, we summarise this equivalence here. For a complete lattice $L$ and any pair $x,y \in L$, $y\succ x$ is the {\bf well above relation} defined by: $y\succ x$
if whenever $x\ge\bigwedge S$, with $S\subseteq L$, there exists
some $s\in S$ such that $y\ge s$. The top and bottom elements of any lattice $L$ will be denoted by $\top_L$ and $\bot_L$, respectively, or simply $\top$ and $\bot$ when $L$ is obvious from context. So as to lighten the notational burden, when we find ourselves in the presence of an indexed family of lattices $L_i$, their respective top and bottom elements' subscripts will be the lattice's index instead. That is, $\bot_i$ and $\top_i$ rather than $\bot_{L_i}$ and $\top_{L_i}$, respectively. A complete lattice $L$ is {\bf completely distributive} provided arbitrary joins distribute over arbitrary meets. A well-known characterisation of completely distributive complete lattices is the following one.

\begin{theorem}[Raney \cite{10.2307/2032165}]\label{thm:raney} A complete lattice $L$ is completely distributive if, and only if, for all $y \in L$ 
\[
y = \bigwedge \{a\in L\mid a \succ y\}.
\]
\end{theorem}
 Some authors (e.g., Flagg) take the above characterisation as the definition of a completely distributive complete lattice. A more modern and thorough work on completely distributive complete lattices by Vickers can be found here \cite{VICKERS1993201}.
 \begin{definition}
A {\bf value distributive lattice}  is a completely distributive complete lattice $L$ for which $L_\prec = \{a \in L \mid a \succ \bot\}$ forms a filter.
A {\bf value quantale} is a value distributive lattice $L$ together with an associative and commutative binary operation $+:L\times L\to L$ such
that for all $x\in L$
\begin{itemize}
\item $x+\bot=x$, and
\item $\bigwedge(x+S)=x+\bigwedge S$ for all $S\subseteq L$.
\end{itemize}
Value quantales will often be denoted with the letters $U, V$ and $W$.
\end{definition}

\begin{definition}
Let $V$ be a value quantale. A {\bf $V$-\emph{space}} is a pair $(X,d)$
with $X$ a set and $d:X\times X\to V$ such that
\begin{itemize}
\item $d(x,x)=\bot$ for all $x\in X$, and
\medskip 
\item $d(x,z)\le d(x,y)+d(y,z)$ for all $x,y,z\in X$.
\end{itemize}

\end{definition}

Given any value quantale $V$ and a $V$-space $(X,d)$ we follow Flagg's adoption of Kopperman's terminology and denote the triple $(V,X,d)$ as a {\bf continuity space}. The category ${\bf M}_{\bf T}$ will be that of all continuity spaces with a morphism $(V,X,d)\to(W,Y,m)$ being a function $f:X\to Y$ such that for
every $x\in X$ and for every $\epsilon\in W_\prec$
there exists $\delta\in V_\prec$ such that for all
$x'\in X$: if $d(x,x')\prec \delta$ then $m(f(x),f(x')) \prec \epsilon$. We call these morphisms {\bf $\epsilon$-$\delta$ continuous functions}. One can readily verify that ${\bf Met}$ is a full subcategory of ${\bf M}_{\bf T}$: every ordinary metric
space $(X,d)$ is a $V$-space for $V=[0,\infty]$ with $+$ being
ordinary addition.

\begin{definition}
Let $(X,d)$ be a $V$-space and $\epsilon\in V$ with $\epsilon\succ \bot$. $B_{\epsilon}(x)=\{y\in X\mid d(x,y)\prec\epsilon\}$ is the {\bf open
ball} with radius $\epsilon$ about the point $x\in X$.
 \end{definition}
 
\begin{theorem}\label{thm:quantaletotopo}
Let $(X,d)$ be a $V$-space. Declaring a set $U\subseteq X$ to be {open}
if for every $x\in U$ there exists $\epsilon\succ \bot$ such
that $B_{\epsilon}(x)\subseteq U$ defines a topology on $X$. Moreover, $\epsilon$-$\delta$ continuity and topological continuity coincide.
\end{theorem}

\begin{proof} The first part is proved in \cite{MR1452402} as Theorem 4.2 and Theorem 4.4 while morphisms are dealt with in \cite{MR3334228}.
\end{proof}

 The functor $\mathcal{O} :{\bf M}_\boldt\to{\bf Top}$ is the one that maps a continuity space to the topology it defines, as detailed in the previous result. In line with standard metrisable topologies, given a continuity space $(V,X,d)$ we will refer to $\mathcal{O}[(V,X,d)]$ as the topology {\bf generated} by $(V,X,d)$. Before we describe $\mathcal{O}$'s inverse we introduce the following: for any collection of sets $\A$ and $A\subseteq \A$, we say that $A$ is {\bf downwards closed} provided that for $B,C \in \A$ and $B\subseteq C$, if $C\in A$ then $B\in A$. Also, we follow standard set-theoretic notation in that for any set $X$, we let $[X]^{<\omega}$ denote the collection of all finite subsets of $X$. Lastly, for any set $X$ let 
\[
\Omega(X)=\{ A\subseteq [X]^{< \omega} \mid A \text{ is downwards closed}\}.
\]
\begin{lemma}[Flagg] Given a set $X$, ordering $\Omega(X)$ by reverse set inclusion yields $(\Omega(X), \leq,+)$ as a value quantale where $+$ is given by intersection and $p \succ \bot$ if, and only if, $p$ is finite.
\end{lemma}
\begin{proof} This is part of Example 1.1 in \cite{MR1452402}.
\end{proof}

The following result is due to Weiss and Flagg and we reproduce the argument here.
\begin{theorem}
$\mathcal{O}:{\bf M}_\boldt\to{\bf Top}$ is an equivalence of categories. 
\end{theorem}

\begin{proof}
In view of Theorem~\ref{thm:quantaletotopo} we need only find an inverse to $\mathcal{O}$ relative to objects. Take any topological space $(X,\tau)$ and construct an $\Omega(\tau)$-space
$(X,d)$ with 
\[
d(x,y)=\{F\in[\tau]^{<\omega}\mid\mbox{for all } U \in F \mbox{ if }  x \in U \mbox{ then } y \in U \}.
\]
\noindent
Let $x\in U \in \tau$, with $U\not=\emptyset$. Denote $\epsilon = \{\emptyset , \{U\}\}$ and observe that
\begin{align*}
y\in B_\epsilon(x) &\Leftrightarrow d(x,y) \prec \epsilon\\
                             &\Leftrightarrow  d(x,y) \supseteq \epsilon \\
                             &\Leftrightarrow y\in U.
\end{align*}
Hence, all non-empty open sets are open balls. Therefore, $\mathcal{O}[(\Omega(\tau), X,d)] = (X,\tau)$.
\end{proof}

\section{Topological notions in ${\bf M}_\boldt$}

We begin by illustrating some basic topological properties in ${\bf M}_\boldt$. Given a continuity space $(X, V, d)$ with $x\in X$ and $C\subseteq X$ we denote 

\[d(x,C) = \bigwedge_{y\in C} d(x,y),\]

\noindent
in line with point-to-set distances in metric spaces. In the following lemma we make use of the fact that given any complete lattice $L$, $A\subseteq L$ and $x\in L$ it follows that $x\succ \bigwedge A$ if, and only if, there exists an $a\in A$ with $x\succ a$ (Lemma 1.3 in \cite{MR1452402}).
 
\begin{lemma}
For a continuity space $(X, V, d)$ with $x\in X$ and $C\subseteq X$:
\medskip
\begin{enumerate}

\item $d(x,C)>\bot$ if, and only if, there exists $\epsilon\succ\bot$ with $B_{\epsilon}(x)\cap C=\emptyset$,
\smallskip
\item the topological closure of $C$ is simply ${\{z\mid d(z,C)=\bot\}}$, and thus
\smallskip
\item $C$ is closed if, and only if, $d(x,C) = \bot \implies c\in C$.
\end{enumerate}

\end{lemma}

\begin{proof} The latter two follow from $(1)$ and to prove that claim assume that $d(x,C)= \bigwedge_{y\in C} d(x,y) = \bot$. Since $V$ is completely distributive, by Theorem~\ref{thm:raney} this happens precisely when for all $\epsilon \succ \bot$, 
\[\epsilon \succ  \bigwedge_{y\in C} d(x,y).\]
Because $V$ is a complete lattice, the fact preceding the lemma yields that for all $\epsilon \succ \bot $ we can find a corresponding $y\in C$ with $d(x,y) \prec \epsilon$. This completes the proof.
\end{proof}

In what follows we adopt standard set-theoretic notation: given a pair of sets $X,Y$ the symbol $X^Y$ denotes the collection of all functions from $Y$ to $X$ and for a product $X= \prod_{i\in I} X_i$, standard projections will be denoted by $\pi_j(X)$ for $j\in I$. For $(V,X,d)$, $C \subseteq X$ and any $R \in V_\prec^X$ we denote
\[
B_R(C) = \bigcup_{x\in C}B_{R(x)}(x).\]

\noindent
 The following are straightforward to verify.

\begin{lemma}\label{lem:subcatsTop} Let $(V,X,d)$ denote a continuity space. 
\medskip
\begin{enumerate}[(a)]
\item  For all $x,y\in X$, $d(x,y)  = \bot = d(y,x) \Longrightarrow x=y$ if, and only if, $\mathcal{O}[(V,X,d)]$ is {Kolmorogov}.
\medskip
\item For all $x,y\in X$, $d(x,y) = \bot \Longrightarrow x=y$ if, and only if, $\mathcal{O}[(V,X,d)]$ is {Fr\'{e}chet}.
\medskip 
\item For all $x,y\in X$ and any closed set $C$, $\bigvee_{R\in V_\prec^X}d(x,B_{R}(C))=\bot \Longrightarrow d(x,C)=\bot$ if, and only if, $\mathcal{O}[(V,X,d)]$ is {regular}.
\medskip
\end{enumerate}
\end{lemma}

Denote ${\bm}_K, {\bm}_F$ and ${\bm}_R$ the full subcategories of ${\bf M}_\boldt$ satisfying the conditions of part (a), (b) and (c) from the above lemma, respectively. It is straightforward to show that the categories $\emph{\bm}_K, \emph{\bm}_F$ and $\emph{\bm}_R$ are reflective subcategories of $\emph{\bm}_\boldt$. As aforementioned, complete regularity also has a metric counterpart: symmetry. However, their equivalence will be neither straightforward as the ones above nor an `if, and only if,' statement. Before we establish this equivalence we will focus momentarily on other basic topological constructions in ${\bf M}_\boldt$.

 \subsection{Subspaces and Products}\label{subsec:limits} Observe that if given a topological space $(X,\tau)$ and a continuity space $(V,X,d)$ with $\mathcal{O}[(V,X,d)] = (X,\tau)$, then for any $Y\subseteq X$ the subspace topology $(Y,\tau_Y)$ is equal to $\mathcal{O}[(V,Y,d_Y)]$ where $d_Y$ is the restriction of $d$ to $Y$. Hence, given another continuity space $(W,Z,m)$, $\epsilon$-$\delta$ continuous functions $f,g:X\to Z$ and $Y$ as the set-theoretic equaliser of $f,g:X\rightrightarrows Z$ we have that $(V,Y,d_Y)$ is the equaliser in ${\bf M}_\boldt$. As will soon become apparent, all other (co)limits will require more delicate constructions. Indeed, consider for example the case of finding the product of two metric spaces: $(\R,X,d)$ and $(\R,Y,m)$. It is well-known what this product looks like. In particular, we know that it can be of the form $(\R, X\times Y, s)$, where $s: (X\times Y)^2\to \R$ is usually given as some combination of $d$ and $m$. When faced with a pair of arbitrary continuity spaces $(V,X,d)$,$(W,Y,m)$ we can only be sure that the underlying set must be isomorphic to $X \times Y$. The value quantale for this product is an entirely different matter. An initial guess could be $V\times W$ ordered pointwise. This, in general, will not work: the well-above elements, in all but the most degenerate of cases, do not form a filter. The reader can verify this simply by computing $[0,\infty]\times[0,\infty]$. As we shall shortly illustrate, and perhaps not surprisingly, a suitable value quantale for the above scenario can be developed from the original suggestion. In what follows, we refer to finite sequences of points in a space as {\bf walks}. We will denote such walks as either $x_1, \ldots, x_n$ or simply $w(x_1, x_n)$, $v(x_1, x_n)$, $u(x_1, x_n)$, etc., where the starting and ending point are $x_1$ and $x_n$, respectively.
 
 \begin{definition} Given a continuity space $(V,X,d)$, a walk $w(x_1, x_n)$ in $X$ and $R \subseteq V_\prec \times X$ we say that $R$ {\bf admits} $w(x_1, x_n)$, and write $R\vdash (x_1, x_n)$, provided that for every $i\leq n-1$ we can find $(\epsilon, x_i) \in R$ with $x_{i+1} \in B_{\epsilon}(x_i)$ for all $i\leq n-1$. 
 \end{definition}
 
 Notice that this new relation admits \emph{concatenation of walks}: $R\vdash w(x_1, x_n)$ and $R\vdash u(y_1,y_m)$ with $x_n = y_1 \Rightarrow R\vdash x_1, \ldots, x_n,y_1, \ldots, y_m$. Consider a family $\{(V_i,X_i,d_i)\}_{ i\in I}$ of continuity spaces, and put $X= \prod X_i$ and $U = [\prod_0 (V_i)_\prec]^X$, where $ \prod_0 (V_i)_\prec$ denotes the collection of all tuples of length $I$ with all but a finite number of coordinates being $\top_i$. As it stands $U$ is not a value quantale, however, all of the information needed to forge the appropriate one lies within it. Put $V = \Omega(U)$ - a value quantale - and fix any pair $a,b\in X$. Given an $R\in U$ we also write $R\vdash(a,b)$ provided that there exists at walk $a=x_1, \ldots, b=x_n$ so that for each $i\in I$ the collection $\{(\epsilon,x) \mid \epsilon = \pi_i(R(x_k))\text{ and } x = \pi_i(x_k) \text{ with } 1\leq k \leq n\}$ admits $\pi_i(x_1),\ldots, \pi_i(x_n)$. Next, put
 \[
 d(a,b) = \{ A\in[U]^{<\omega} \mid \forall R\in A,\exists w(a,b) \mbox{ with } R \vdash w(a,b)\}.
 \]
 \noindent
 Notice that $d(a,b)$ is closed downwards and thus $d(a,b) \in V$. Also notice that $d(x,x) = \bot$ for all $x\in X$. Also, given a triplet $x,y,z\in X$: $d(x,y)+d(y,z) = d(x,y) \cap d(y,z) \subseteq d(x, z)$ (by concatenation of paths). Consequently, $d(x,y) + d(y,z) \geq d(x,z)$ and $(V,X,d)$ is a continuity space. 
\begin{theorem} For a family $\{(V_i,X_i,d_i)\}_{ i\in I}$ of continuity spaces and $(V,X,d)$ as defined above we have that 
\[
(V,X,d) = \prod_{i \in I} (V_i,X_i,d_i). 
\]
That is, $\mathcal{O}[(V,X,d)]$ is the product of the family  $\{\mathcal{O}[(V_i,X_i,d_i)]\}_{ i\in I}$.
\end{theorem}
 \begin{proof} First we show that $\mathcal{O}[(V,X,d)]$ is at least as fine as the product space. Fix a $j\in I$, an $x\in X_j$, and an $\epsilon \in V_j$, and put $B = \pi_j^{-1}(B_\epsilon(x))$. Choose $f_\epsilon  \in [\prod_0 (V_i)_\prec]^X$ so that when given a $z\in B$, $\pi_j(f_\epsilon(z)) = \delta$ with $B_\delta(\pi_j(z))\subseteq B_\epsilon(x)$ and when $i\not = j$, $\pi_i(f_\epsilon(z)) = \top_i$. For all $z\not \in B$ the choice is irrelevant but we make one nonetheless: let $\pi_i(f_\epsilon(z)) = \top_i$. Lastly, fix any $p\in X$ with $\pi_j(p) = x$ and observe that $B_{\{f_\epsilon\}}(p) = B$.
 
 To prove the reverse containment, begin by picking any $\epsilon \succ \bot_V$ and fixing any $p \in X$. Notice that since $|\epsilon|$ is finite, the point-wise meet of all functions in $\cup{\epsilon}$ - a finite number of them - defines a function in $[\prod_0 (V_i)_\prec]^X$: denote it by $f$. Let $F\subseteq I$ represent the finite set of indices $j$ where $\pi_j[f(p)] \not = \top_j$; indeed, $F$ can be empty. In any case, the basic open set 
 \[
 \bigcap_{j\in F} \pi_j^{-1}\left[B_{\pi_j[f(p)]}(\pi_j(p))\right]
 \]
 in the product topology is contained in $B_\epsilon(p)$. Thus, the product topology is at least as fine as the one generated by $(V,X,d)$ and the proof is complete.
 \end{proof}
 
 \begin{remark} The metric construction of the box topology is identical to the one outlined above with the slight difference that $U = [\prod(V_i)_\prec]^X$. It is also worth noticing that the method of walks, which we recycle in the sequel when dealing with colimits, could have been employed when describing subspaces, thus illustrating a general theme among all topological (co)limits. 
 
 Another striking motif among these constructions - already partially witnessed - is the similarity between the underlying sets and their value quantale counterpart. This observation will be reinforced in what follows.
 \end{remark}

\subsection{Sums and Quotients} Consider a family $\{(V_i,X_i,d_i)\}_{ i\in I}$ of continuity spaces, and put $X= \coprod X_i$ and $U = [\coprod (V_i)_\prec]^X$, where $\coprod$ denotes the usual disjoint union. Turn $U$ into a value quantale by letting $V = \Omega(U)$ and, again, for a pair $a,b\in X$ let 
 \[
 a\triangle b = \{ A\in[U]^{<\omega} \mid \forall R\in A,\exists w(a,b) \mbox{ with } R \vdash w(a,b)\}.
 \]

\noindent
Notice that in the above definition no finite walk in $X$ with elements from different factors of $X$ is admitted by any $R\in U$. That is, for $a\in X_k$ and $b\in X_j$ with $j\not = k$ we have $a\triangle b = \emptyset = \top_V$. Consequently, for all $a,b\in X$ we have $a\triangle b \in V$ and we, thus, define $d:X^2 \to V$ by $d(a,b) = a\triangle b$. Again, $d(a,a) = \bot$ for all $a\in X$ and concatenation of paths yields the triangle inequality. By design, for a fixed $i\in I$ the topology generated by the restriction of $d$ onto $X_i$ matches exactly the one generated by $(V_i, X_i,d_i)$ and, as one would expect, points from different factors of $X$ are as far away from each other as possible. Thus, the V-space $(X,d)$ clearly generates the topological sum on $X$.
\begin{theorem} For a family $\{(V_i,X_i,d_i)\}_{ i\in I}$ of continuity spaces and $(V,X,d)$ as defined above we have that 
\[
(V,X,d) = \coprod_{i \in I} (V_i,X_i,d_i). 
\]
\end{theorem}

We complete this section by illustrating quotients; the reader versed with the previous constructions should find the following rather familiar. Begin with a continuity space $(Y,W,m)$ and an equivalence relation ${\sim} \in Eq(Y)$. Denote $X = Y/ {\sim}$. We seek a value quantale $V$ and a distance assignment $d:X^2 \to V$ that will generate that quotient topology on $X$. As with equalisers, the value quantale $V$ will be based on $U = (V_\prec)^X$. Put $V = \Omega\left(U\right)$ and let $a\triangle b \in V$ be the collection of $A\in [U]^{<\omega}$ so that for any $R\in A$ there exists a walk $a=x_1, \ldots, b=x_n \in X$ with $x_i\sim x_{i+1}$ provided $i$ is even and $x_{i+1} \in B_{R(x_i)}(x_i)$ when $i$ is odd. As expected, for a pair $a\sim b$ we have $a \triangle b = \bot_V$ and, yet again, concatenation of paths yields the triangle inequality. Finally, letting $d(a,b) = a\triangle b$ one can readily verify that $(V,X,d)$ is a continuity space. Let $q:Y \to X$ be the quotient function giving rise to $\sim$, choose any $x\in X$ and $\epsilon \succ \top_V$. Put

\[
\delta = \bigwedge_{R\in \cup \epsilon} R(x)
\]

\noindent
and notice that $\delta \in W$ since $\cup \epsilon$ is finite. By design, $q\left[B_\delta^m(x)\right] \subseteq B_\epsilon^d(q(x))$ and the latter $\epsilon$-ball is a saturated open set. Consequently, the $V$-space $(X,d)$ generates the quotient topology on $X$.

 \begin{theorem} For a continuity space $(W,Y,m)$ and ${\sim} \in Eq(Y)$ the continuity space $(V,X,d)$ as defined above generates the quotient space.
 \end{theorem}

\begin{remark} Notice that when ${\sim} = \emptyset$ the previous theorem states that $\mathcal{O}(W,Y,m) = \mathcal{O}(V,X,d)$. That is, defining open sets by means of walks is equivalent to the standard use of $\epsilon$-balls. This rather obvious fact is really what inspired the constructions for the above illustrated (co)limits.
\end{remark}

\section{Symmetry and Complete Regularity}

Next we show that the full subcategory ${\bm}_S$ of ${\bm}_\boldt$ comprised of all symmetric continuity spaces is equivalent to the category ${\bf CReg}$, of completely regular topologies. In contrast with the equivalences established for Kolmorogov, Fr\`{e}chet and regular spaces, complete regularity does not allow for an ``if and only if'' statement. This shouldn't come as a surprise: discrete topologies are completely regular and plenty of non-symmetric continuity spaces give rise to discrete topologies. That said, of the equivalence ${\bm}_S \leftrightarrows {\bf CReg}$ we show that the functor ${\bm}_S \to {\bf CReg}$ coincides with the restriction of $\mathcal{O}$ to ${\bm}_S$.

The equivalence ${\bm}_S \leftrightarrows {\bf CReg}$ is not an unexpected result: symmetry in Kopperman's value semigroups corresponds to complete regularity. Moreover, uniformities yield only completely regular spaces and any completely regular topology has a corresponding uniform space. Dyadics are frequently employed when showing complete regularity. What we develop next are suitable embeddings of such within value quantales. The proof of the following can be found in \cite{MR1452402} pg. 264.

\begin{lemma}[Flagg] For a value quantale $(V,\leq, +)$ if $\epsilon \succ \bot$ then there exists $\delta \succ \bot$ so that $\epsilon \succ 2\delta$.
\end{lemma}

This is key to show complete regularity. Put $\D=\{\frac{{i}}{2^{j}} \mid i,j\in\mathbb{N}\}$ and take any $\epsilon \succ \bot$.
We are guaranteed at least one $\delta_{1}\succ\bot$ so that $2\delta_{1}\prec\epsilon$.
Similarly, since $\delta_{1}\succ \bot$ we can find $\delta_{2}\succ \bot$
so that $2\delta_{2}\prec\delta_{1}$ and so on. Thus, we have
$\{\delta_{i}\mid i \in \N \}$ so that $2\delta_{i+1}\prec\delta_{i}$. Next, we define ${ \frac{{\epsilon}}{2^{i}}:=\delta_{i}}$
(where we let $\delta_{0}:=\epsilon$). Clearly, $\frac{{\epsilon}}{2}+\frac{{\epsilon}}{2} = 2\frac{{\epsilon}}{2}\prec\epsilon$
and equality is not at all guaranteed: take the lattice $\{\bot,\top\}$ as an example. At this stage all powers of $\frac{1}{2}$ have been suitably defined. Representing any other dyadic multiple of $\epsilon$ is then simple: any dyadic number smaller than $1$ can be expressed uniquely as a sum of
products of $\frac{1}{2}$. Hence, given
any $n\in \D$ we can express $n= \left \lfloor{n}\right \rfloor + \sum_{i=1}^{\infty}\frac{f_n(i)}{2^{i}}$
where  $f_n(i) \in \{0,1\}$ and $f_n(i) = 0$ for sufficiently large $i$. In
general we let, for $n\not=\frac{1}{2^{i}}$ (since those are already defined) $n\epsilon:=\left \lfloor{n}\right \rfloor\epsilon + \sum_{i=1}^{\infty}\frac{f_n(i)}{2^{i}}\epsilon$. It
is simple to show that for $n, m\in \D$ we have $m\epsilon + n\epsilon \leq (m+n)\epsilon$, and if $n\leq m$ then $n\epsilon \leq m\epsilon$.

\begin{lemma}\label{lem:dyadic}  Let $V$ be a value quantale, $\epsilon \in V_\prec$ and let $f:\D \to V_\prec$ with:
\begin{itemize}
\item $f(1) : = \epsilon$, $f(0) = \bot$ and for each $i\in \N$, $2f\left(\frac{1}{2^{i+1}}\right) \prec f\left(\frac{1}{2^i}\right)$;
\item for any other $n\in \D$, put
\[
f(n) :=\left \lfloor{n}\right \rfloor\epsilon + \sum_{i=1}^{\infty}f\left[\frac{f_n(i)}{2^{i}}\right]\]

with $n= \left \lfloor{n}\right \rfloor + \sum_{i=1}^{\infty}\frac{f_n(i)}{2^{i}}$.
\end{itemize}

It follows that $f$ is order preserving and that for any pair $n, m\in \D$ we have $f(m) + f(n) \leq f(m+n)$.
\end{lemma}
The above result then allows for the implementation of Kopperman's original argument based on his continuity spaces (see~\cite{MR935419}).

\begin{lemma}
Any symmetric continuity space yields a completely regular topology.
\end{lemma}

\begin{proof}
Take any symmetric continuity space $(V,X,d)$, with $V=(V,\leq_V,+_V)$ and fix an $x_0\in O\in\tau_{V}$
(i.e. the topology generated by $(V,X,d)$) and notice that there
exists an $\epsilon\succ \bot$ so that $\overline{B_{\epsilon}(x)}\subseteq O$
(that is, there is a closed set of radius $\epsilon$ about $x_0$
entirely contained within $O$). Indeed, since $x_0\in O$ then there exists a $p \succ \bot$ for which $B_p(x_0) \subseteq O$. Since $p\succ \bot$ we can find $\epsilon \succ \bot$ for which $2\epsilon \prec p$. In turn, $d(x_0,y) \leq \epsilon \Rightarrow d(x_0,y) \prec p$ and $y\in B_p(x_0)$. For this $\epsilon$ we can construct an order preserving $f: \D \to V_\prec$ with the properties illustrated in Lemma~\ref{lem:dyadic}; for each $n\in \D$ we write $n\epsilon = f(n)$. Let us fix one such copy and continue with the proof. Put $M_{\epsilon}:V\to[0,\infty]$
with 
\[
M_{\epsilon}(a)= \text{inf}\{n\in \D\mid a\leq n\epsilon\} 
\]

\noindent
and observe that $M_\epsilon$ preserves the triangle inequality relative to $[0,\infty]$. Indeed, if $c\leq_V a+_Vb$, and $a\leq_V r\epsilon$ and $b\leq_V s\epsilon$ for a pair $r,s \in \D$ then
$c\leq_V a+_Vb\leq_V r\epsilon+_Vs\epsilon\leq_V(r+s)\epsilon$. Hence, $r+s \in \{n\in \D\mid c\leq n\epsilon\}$ and
\[
M_{\epsilon}(c)\leq \text{inf}\{ n+ m \mid a\leq n\epsilon \text{ and } b\leq m\epsilon\} = M_{\epsilon}(a)+M_{\epsilon}(b).
\]
Next we define the auxiliary function $g:X\to[0,1]$ by 
\[
g(y)=\text{min}\{M_{\epsilon}(d(x_0,y)),1\}
\]
 and proceed to that show $g$ is continuous; the actual point-closed set separating function is defined subsequently. First notice that for any $y,z\in X$, since $d(x_0,z)\leq d(x_0,y)+d(y,z)$ we have $M_{\epsilon}(d(x_0,z)) \leq M_{\epsilon}(d(x_0,y)) + M_{\epsilon}(d(y,z))$ and, consequently, $g(z)\leq g(y)+M_{\epsilon}(d(y,z))$. It follows that $g(z)-g(y)\leq M_{\epsilon}(d(y,z))$ and, by symmetry of $d$, that
$|g(y)-g(z)|\leq M_{\epsilon}(d(y,z)).$ 

 For continuity of $g$ we require that for any $x\in X$ and any $p>0$ we can find $\delta \succ \bot$ so that $d(x,y)\prec \delta \Rightarrow |g(x) - g(y)| <p$. Choose any $p>0$ and
take any $n\in \mathcal{D}$ so that $n<p$. Notice that if $d(y,z)\prec n\epsilon$ then $|g(y)-g(z)|\leq M_{\epsilon}(d(y,z))\leq n<p$ and thus
$g$ is continuous. Lastly, define $f:X\to[0,1]$ as $f(x)=\text{max}\{0,1-g(x)\}$.
Hence, $f$ is continuous with $x\mapsto1$ and $y\mapsto 0$ for any $y\not\in\overline{B_{\epsilon}}(x)$. \end{proof}

This then establishes the functor ${\bm}_S \to {\bf CReg}$ as a restriction of $\mathcal{O}$ to ${\bm}_S$. A completely regular space coincides with the initial topology of its collection of all continuous real-valued functions. Equivalently, it is also the initial topology of its collection of continuous functions that separate points from closed sets. Indeed, let $\tau$ be any completely regular topology on a set $X$ and denote $\tau^*$ to be the one generated from all of the $\tau$-continuous functions that separate points from closed sets. Let $f:X \to \R$ be any $\tau$-continuous function and fix any $(\alpha,\beta) \subseteq \R$. Denote $(\alpha,\beta)_f = f^{-1}(\alpha,\beta)$ and notice that for any $x \in (\alpha,\beta)_f$ we can find a $\tau$-continuous (and thus $\tau^*$-continuous) point-closed set separating function $g_x: X \to [0,1]$ for which $g_x(x) = 0$ and $f(y) = 1$ for any $y\in X \smallsetminus (\alpha,\beta)_f$. It is then simple to observe that 
$$ \bigcup_{x\in (\alpha,\beta)_f} g_x^{-1}[0,1) = (\alpha,\beta)_f,$$
\noindent
an open set in $\tau^*$. Hence, $f$ is also $\tau^*$-continuous. The idea behind the following theorem is to construct a symmetric continuity space based on such a collection. The construction relies on making all such real-valued functions continuous and generating the original completely regular topology.

\begin{lemma}
Any completely regular topology can be generated by a symmetric continuity
space.
\end{lemma}

\begin{proof}

Let $(X,\tau)$ be a completely regular topology and put $F$ as the collection of all continuous $[0,1]$-valued
functions on $X$ that separate points from closed sets. Let $V=[0,1]^{F}$ and 
\[K=\{f\in(0,1]^{F}\mid f(x)=1\mbox{ for all but finitely many }x\in F\}.
\]
As it stands, neither $V$ nor $K$ are value quantales (the well-above
0 elements do not form a filter). What we do,
by combining Kopperman's and Flagg's ideas, is to embed $V$ into a suitable
value quantale where the well above $\bot$ elements will be exactly the
images of $K$. For a pair $f,g\in V$, we say that $g$ is {\bf way-above}
$f$, $f \ll g$, if $g(x)>f(x)$ for all $x$ for which $f(x)<1$. A {\bf round filter} (cf. \cite{MR1452402} pg. 275)
$p\subseteq K$ is one for which:

\begin{itemize}
\item $\top\in p$,
\smallskip
\item for all $f,g\in K$ if $f\in p$ and $f\ll g$ then $g\in p$, and
\smallskip
\item for any $f\in p,\exists g\in p$ so that $g\gg f$.
\end{itemize}

\noindent
Following Flagg's notation, we let $\Gamma(K)$ denote the collection
of all round filters on $K$; Flagg shows that $\Gamma(K)$ is indeed
a value quantale when ordered by reverse inclusion and addition is
taken to be intersection. 
Next, we
embed $V$ within $\Gamma(K)$ and generate the desired
topology using $\Gamma(K)$. The embedding which we denote by $\Psi$, is the following: for
any $f \in V$, $f\mapsto\hat{f} = \{g\in K \mid f\ll g\}$ and notice that $\wedge\hat{f}=f$ so
that $\Psi$ is an embedding. Moreover, $p\succ \bot$ if, and only if, $p = \hat{f}$ with $f\in K$.

In order to simplify what remains to be proved we define two functions:
one goes from $X\times X$ into $V$ and generates the other function
(the actual metric on $X$) from $X\times X$ into $\Gamma(K)$. The
first function $m:X\times X\to V$ is done coordinate-wise. That is,
for any pair $x,y\in X$ and $f\in F$ we have $m(x,y)(f)=|f(x)-f(y)|$. Obviously, the indiscrete topology is trivially completely regular. We treat this pathological case by letting all distances equal $0$. The second one works as follows:
$d:X\times X\to\Gamma(K)$ so that 
\[
(x,y)\mapsto\Psi(m(x,y))=\widehat{m(x,y)}.
\]
Next we show that the symmetric continuity space $(\Gamma(K),X,d)$ generates $(X,\tau).$ First
we show that for any $p\in\Gamma(K)_\prec$, $B_{p}(x)\in\tau$ for
any $x\in X$. Notice that 
\begin{align*}
y\in B_{p}(x) &\Leftrightarrow d(x,y)\prec p\\
			 &\Leftrightarrow \widehat{\wedge p}\subseteq d(x,y)\\
			 &\Leftrightarrow \wedge p\gg m(x,y).
\end{align*}

\noindent
Recall that $p = \hat{f}$ with $\wedge \hat{f} = f\in K$. Enumerate the finitely many functions $h_i\in F$ with $p(h_i) < 1$ with an index $i \leq n$. By design, $ f \gg m(x,y)$ if, and only if, $p_{i} = f(h_i) >m(x,y)(h_{i})=|h_{i}(x)-h_{i}(y)|$, for all $i \leq n$. In turn we get  $y\in B_{p}(x)$ if, and only if, $h_{i}(x)-p_{i}<h_{i}(y)<h_{i}(x)+p_{i}$ for all $i\leq n$, and the latter happens precisely when $y\in\bigcap_{1}^{n}h_{i}^{-1}[(h_{i}(x)-p_{i},h_{i}(x)+p_{i})]$.
Thus, $B_{p}(x)$ is indeed open in $\tau$. 

Lastly, take $x\in O\in\tau$ and any $f\in F$ so that $f(x)=1$
and $f(y)=0$ for all $y\not\in O$. Here is where it becomes clear that we need only look at continuous functions that separate points from closed sets. Let $h\in K$ so
that $1>h(f)>0$ and $h(g)=1$ for all $g\not=f$. If $y\in B_{\hat{h}}(x)$
then $d(x,y)\prec\hat{h}$ yielding that $\hat{h}\subseteq d(x,y)$. Consequently, 
for $m(x,y)(g)<1$ we get $h(g)>m(x,y)(g)\Rightarrow h(f)>m(x,y)(f)=|f(x)-f(y)|$. By design, $h(f)<1$ and thus $ f(y)>0$. Hence, $ y\in O$ and $B_{\hat{h}}(x) \subseteq O$.
\end{proof}

Lemmas 18 and 19 then establish the equivalence ${\bm}_S \leftrightarrows {\bf CReg}$. Observe that unlike ${\bm}_S \to {\bf CReg}$, the functor ${\bm}_S \leftarrow {\bf CReg}$ does not coincide with the inverse of $\mathcal{O}$. More precisely, given a completely regular topology $(X,\tau)$ the symmetric continuity space $(\Gamma(K),X,d)$, as in Lemma 19, is often different from $(\Omega(\tau),X,d_\tau)$ from Theorem 7. That said, we have the following result.

\begin{corollary} The category $\text{{\bf M}}_S$ is equivalent to {\bf CReg} and a reflective subcategory of $\emph{\bm}_\boldt$ . 
 \end{corollary}
 
\begin{proof} The first part was established by the preceding lemmas. As for the reflector ${\bm}_\boldt \to {\bm}_S$, fix any continuity space $(V,X,d)$ and let $(X,\tau)$ be the reflection of $\mathcal{O}(V,X,d)$ given by the left adjoint to the inclusion functor ${\bf CReg} \hookrightarrow {\bf Top}$. We claim that the assignment  $(V,X,d) \mapsto ((\Gamma(K),X,s), \text{id}_X)$ - $(\Gamma(K),X,s)$ as in Lemma 19 and $\text{id}_X$ denoting the identity on $X$ - defines the reflector of $\mathcal{O}$. Clearly, $\text{id}_X: (V,X,d) \to (\Gamma(K),X,s)$ is $\epsilon$-$\delta$ continuous.

Let $(W,Y,m)$ be an arbitrary symmetric continuity space and $f:(V,X,d) \to (W,Y,m)$ be $\epsilon$-$\delta$ continuous. It follows that $f:\mathcal{O}(V,X,d) \to \mathcal{O}(W,Y,m)$, $\text{id}_X: \mathcal{O}(V,X,d) \to (X,\tau)$ and $f: (X,\tau) \to \mathcal{O}(W,Y,m)$ are all continuous. Since $\mathcal{O}(\Gamma(K),X,s) = (X,\tau)$ the equivalence $\mathcal{O}: \textbf{M}_\textbf{T} \to \textbf{Top}$ yields that $f:(\Gamma(K),X,s) \to (W,Y,m)$ is $\epsilon$-$\delta$ continuous. This completes the proof.
\end{proof}
There exist at least three obvious ways to \emph{symmetrise} a continuity space (and thus `completely regularise' a topology): given a topology $\tau$, which is not completely regular, we can generate its corresponding continuity space, say $(V,X,d)$. Since $d$ is not symmetric (for then $\tau$ would be completely regular) its dual $d^*: X \times X \to V$ so that $d(x,y) = d^*(y,x)$ for all $x,y \in X$, differs from $d$. Let $\tau^*$ be the topology generated by $(V,X,d^*)$. We then have three obvious candidates for a \textit{symmetrisation} of $(V,X,d)$: 
\medskip
\begin{enumerate}[(a)]
\item $(V,X,d_\vee)$ so that $m_\vee(x,y) = d(x,y) \vee d^*(x,y)$,
\medskip
\item  $(V,X,d_+)$ so that $m_+(x,y) = d(x,y) + d^*(x,y)$ and
\medskip 
\item $(V,X,d_\wedge)$ so that 
\[ 
d_\wedge (x,y) = \bigwedge_\gamma \left(\sum^n_i d(a_i,a_{i+1}) \wedge d^*(a_{i},a_{i+1})\right)
\]

\noindent
where the meet is indexed over all walks $\gamma$ of the form $x=a_1, \ldots, y=a_n$.

\end{enumerate}
\medskip
Flagg proves that (b) generates $\tau \vee \tau^*$. The same is true for (a): clearly, $\mathcal{O}(V,X,d_\vee)\geq \tau \vee \tau^*$ and since $\forall p \succ \bot$, $B_p^m(x) = B_p^d(x) \cap B_p^{d^*}(x)$ then $\mathcal{O}(V,X,d_\vee)\leq \tau \vee \tau^*$. Thus, $\mathcal{O}(V,X,d_\vee) = \mathcal{O}(V,X,d_+)$. Notice that in (c), $\mathcal{O}(V,X,d_\wedge)\leq \tau \wedge \tau^*$ and equality is not at all guaranteed; see \cite{MR3340269} Theorem 4.2 (b). 

None of these symmetrisations is equivalent to the left adjoint of ${\bf CReg} \hookrightarrow {\bf Top}$. Since (a) and (b) generate topologies finer than the original one the claim is then clear for such cases. In contrast, (c) is not as straightforward. Let $X = \{x\} \cup \{x_n\}_{n\in \N}$ and $d: X \times X \to \R$ so that:
\[
d(y,z) = \begin{cases} \frac{1}{n} &\mbox{if } y=x_n \mbox{ and  } z=x\\ 
1 & \mbox{otherwise.}
 \end{cases}
\]

\noindent
The functor $\mathcal{O}$ maps $(V,X,d)$ to the discrete topology on $X$ (a completely regular topology) but the symmetrisation of $(X,d)$ via (c) is mapped by $\mathcal{O}$ to a strictly coarser topology than the discrete topology on $X$. Indeed, the sequence $(x_n)$ conveges to $x$ in $(V,X,d_\wedge)$, thus (c) can not even define a left adjoint to the inclusion ${\bm}_S \hookrightarrow {\bm}_\boldt$.

\begin{question} What is the equivalent to the left adjoint of ${\bf CReg} \hookrightarrow {\bf Top}$?
\end{question}

\section{Conclusion}

We conclude this manuscript by reestablishing some classical results through this metric formalism and highlighting their almost immediate evidence. The metric versions of separation axioms $T_0$, $T_1$ and $T_3$ are easily shown to be productive. As illustrated in Section~\ref{subsec:limits} it is clear from the construction of products in ${\bm}_\boldt$ that a given product is symmetric precisely when all of its factors are also. Thus, in conjunction with the equivalence ${\bm}_\boldt \leftrightarrows \bf{Top}$ we can easily establish the following.

\begin{theorem} Any product of continuity spaces is symmetric precisely when all of its factors are also. Consequently, the same is true for Tychonoff spaces.
\end{theorem}

Affine to our work is a recent publication by Weiss (see \cite{MR3482717}) establishing a metric, and \emph{positive}, characterisation of connectedness. 

\begin{theorem}[Weiss] A continuity space $(V,X,d)$ generates a connected topological space if, and only if, for all $a,b\in X$ and for all $R\in V_\prec$ there exists a walk $a=x_1, \ldots, b=x_n$ with either $d(x_i,x_{i+1}) \prec R(x_i)$ or $d(x_{i+1},x_{i}) \prec R(x_{i+1})$.
\end{theorem}

\begin{corollary} Any product of continuity spaces is connected precisely when all of its factors are as well. 
\end{corollary}
\begin{proof} By construction of the well-above relation in the product space, the failure of the product space (resp. any one of the factor spaces) to satisfy connectedness implies the failure of at least one its factor spaces (resp. the product space) also.

\end{proof}

\section*{Acknowledgments}
We would like to extend our sincere gratitude to Paul Szeptycki for his many helpful suggestions, in particular with the metric construction of (co)limits.

\bibliographystyle{plain}
\bibliography{mybib}

\end{document}